\documentclass[12pt]{amsart}
\usepackage{amsmath, amsthm, amssymb, amsfonts}
\usepackage{tikz}
\usepackage[utf8]{inputenc}
\usepackage[OT2,T2A]{fontenc}
\usepackage{graphicx}
\newcommand{\RR}{\mathbb{R}}

\newtheorem{theorem}{Theorem}

\theoremstyle{definition}
\newtheorem{definicja}{Definition}
\theoremstyle{definition}

\tikzset{
	partial ellipse/.style args={#1:#2:#3}{
		insert path={+ (#1:#3) arc (#1:#2:#3)}
	}
}
\usetikzlibrary{arrows.meta}
\newcommand {\bezierq}[5]{
	\newdimen\pxa
	\newdimen\pya
	\newdimen\pxb
	\newdimen\pyb
	\newdimen\pxc
	\newdimen\pyc
	\newdimen\pxd
	\newdimen\pyd
	\newdimen\pxe
	\newdimen\pye
	\pgfextractx{\pxa}{#1}
	\pgfextracty{\pya}{#1}
	\pgfextractx{\pxb}{#2}
	\pgfextracty{\pyb}{#2}
	\pgfextractx{\pxc}{#3}
	\pgfextracty{\pyc}{#3}
	\pgfextractx{\pxd}{#4}
	\pgfextracty{\pyd}{#4}
	\pgfextractx{\pxe}{#5}
	\pgfextracty{\pye}{#5}
	\foreach \t in {0.05,0.1,...,1}{
		\pgfmathsetmacro{\pxf}{\pxa*(1-\t)^4 + \pxb*4*\t*(1-\t)^3 + \pxc*6*\t^2*(1-\t)^2 + \pxd*4*(1-\t)*\t^3 + \pxe*\t^4}
		\pgfmathsetmacro{\pyf}{\pya*(1-\t)^4 + \pyb*4*\t*(1-\t)^3 + \pyc*6*\t^2*(1-\t)^2 + \pyd*4*(1-\t)*\t^3 + \pye*\t^4}
		\pgfmathsetmacro{\q}{\t-0.05}
		\pgfmathsetmacro{\pxi}{\pxa*(1-\q)^4 + \pxb*4*\q*(1-\q)^3 + \pxc*6*\q^2*(1-\q)^2 + \pxd*4*(1-\q)*\q^3 + \pxe*\q^4}
		\pgfmathsetmacro{\pyi}{\pya*(1-\q)^4 + \pyb*4*\q*(1-\q)^3 + \pyc*6*\q^2*(1-\q)^2 + \pyd*4*(1-\q)*\q^3 + \pye*\q^4}
		\draw (\pxi pt,\pyi pt)--(\pxf pt,\pyf pt);
	}
}

\title[A generalization of Clairaut's formula\dots]{A generalization of Clairaut's formula and its applications}
\author{Vadym Koval}
\address{Jagiellonian University, Faculty of Mathematics and Computer Science, \L ojasiewicza 6, 30-348 Krak\'ow, Poland}
\email{vadym.koval@student.uj.edu.pl}
\date{June 27th 2023, revised September 28th 2023}

\begin{document}
	
	\maketitle

	\begin{abstract}
		The main purpose of this article is to study conditions for a curve on a submanifold $M\subset\mathbb{R}^n$, constructed in a particular way involving the Euclidean distance to $M$, to be a geodesic. We also present the naturally arising generalization of Clairaut's formula needed for the generalization of the main result to higher dimensions.
	\end{abstract}
	
	\section{Introduction}
	
	The motivation for this article comes from Hardt's conjecture. In 1982 R. Hardt posed a hypothesis in \cite{Hardt} which is still open. In order to state it properly, we introduce the next definitions.
	\begin{definicja}
		A subset $S\in \RR^n$ is called \textit{semi-algebraic}, if it is a finite sum of sets $S_k$, where each set $S_k$ is the set of solutions of a finite system of polynomial equations and inequalities.
	\end{definicja}
	\begin{definicja}
		A function $f:A\to B$, where $A\subset \RR^n$, $B\subset 
		\RR^m$ is called semi-algebraic if its graph is a semi-algebraic subset of $\RR^{n+m}$.
	\end{definicja}
	
	\noindent From the general theory we know that a projection of a semi-algebraic set is semi-algebraic (Tarski-Seidenberg Theorem), therefore the domain of a semi-algebraic function is semi-algebraic as well. Moreover, a connected semi-algebraic set is path connected with rectifiable paths. More on that topic can be found e.g. in the survey \cite{long}.
	
	\noindent Consider a connected semi-algebraic set $X$ endowed with the metric $d$ defined by the length of the shortest paths connecting two points in that set. Formally, \begin{align*}d(x,y)=\inf \{ lg(\gamma)|\gamma:[0,1]\to X, &\textrm{a rectifiable path such that}\\ &\gamma(0)=x,\gamma(1)=y\},\end{align*} where $lg(\gamma)$ is the length of the continuous curve $\gamma$. From the general theory it follows that the definition is well-posed. Hardt's conjecture in its simplest form says that such a function $d$ is semi-algebraic (to be more precise, the question pertains to \textit{subanalytic geometry} which is a generalization of semi-algebraic geometry, and was stated in this category, but to make it accessible to a larger audience we restrict ourselves to the semi-algebraic case, which is easier to present). As a matter of fact, the best result obtained in this context up to now is due to Kurdyka and Orro \cite{KO}. In the present article we will discuss one possible approach to this problem in the case of a smooth manifold. Note that a semi-algebraic or subanalytic set is smooth at its generic point.

	\bigskip
	This approach relies on a special but natural local construction of curves on a submanifold. The construction presented in the next section seems at first to give always a geodesic but as it eventually turns out, there is a simple counter-example, see section \ref{cex}. On the other hand, when examining the counter-example we are led to an interesting generalization of Clairaut's formula which we give in the last section.
	
	\section{A natural construction}
	
	A natural approach to Hardt's conjecture is based on the idea to constructively describe the shortest path between two points of the set, at least locally. Let us begin with a two dimensional smooth submanifold $M$ of $\RR^3$ and points $A$ and $B$ in it which we want to connect by a geodesic. We perform the following construction: we connect points $A$ and $B$ by a segment in the ambient space and for each point $x$ of the segment $[A,B]$ we choose a point $m(x)$ realizing the Euclidean distance to $M$. As observed e.g. in \cite{Szkielety}, the vector $x-m(x)$ is perpendicular to $M$ at $m(x)$, i.e. it belongs to the normal space of $M$ at $m(x)$. Is it true that the described curve $m([A,B])$ (if it indeed is a curve) is the shortest path (or at least geodesic) on $M$? We will answer that question asked by M. Denkowski.\\
	Note that the question about geodesicity is natural, as geodesics are `locally shortest' curves on the surface. For the sake of completeness we give a definition of a geodesic in our setting.
	
	\begin{definicja}
		A smooth curve $\gamma(t)$ on a submanifold $M$ is called a \textit{geodesic} if $||\gamma'(t)||=1$ (i.e. the curve is parametrized by its length) and $\gamma''(t)\perp T_{\gamma(t)}M$ (i.e. its acceleration is perpendicular to the tangent space).
	\end{definicja}
	
	\noindent We will note that the definition mentioned above makes sense already for twice differentiable curves on manifolds $M\subset\mathbb{R}^n$ of class $\mathcal{C}^2$. Moreover, the perpendicularity condition itself implies that the speed is constant, as $\frac{d}{dt}||\gamma'(t)||^2=2\langle \gamma'(t), \gamma''(t)\rangle=0$, because $\gamma'(t)\in T_{\gamma(t)}M$. The definition only normalizes the speed.
	
	\vspace{5mm}
	
	\noindent Let us understand when the construction described above makes sense. First of all, for the curve to be well-defined we need to know that the closest point is unique. We naturally arrive at the assumption that the segment $[A,B]\subset\mathbb{R}^3$ must be disjoint with the so-called \textit{medial axis} of the set $M$.
	\begin{definicja}
		The \textit{medial axis} $S_X$ of a closed set $X$ in $\RR^n$ is a set of those points of the space which have more than one point realizing the (Euclidean) distance to the set. Formally, $$S_X=\{x\in \RR^n|\exists u, v\in X: u\neq v, \: dist(x,X)=||x-u||=||x-v||\}.$$
	\end{definicja}
	\noindent More about medial axes can be found in \cite{Szkielety}. In particular, the Nash Lemma mentioned there guarantees the existence of a neighborhood of $M$ disjoint with $S_M$ in the case when $M$ is a $\mathcal{C}^k$-submanifold with $k\geq 2$. Moreover, the neighborhood can be chosen such that the function $x\mapsto m(x)$ is of class $\mathcal{C}^{k-1}$. Unfortunately, such a neighborhood may not (and typically will not) be convex. Of course, it is possible at least locally to connect points on $M$ with segments lying entirely in such a neighborhood.\\
	The continuity of the constructed curve could be obtained directly. It luckily happens (a proof can be found in \cite{continuity}, Theorem 3.3, see also \cite{Szkielety}) that we get it for free from the disjointedness of $[A,B]$ with the medial axis. From the above-mentioned theorem of Tarski-Seidenberg it follows that the constructed curve is semi-algebraic when $M$ is such (and subanalytic in the subanalytic case, cf. \cite{long}).
	\vspace{5mm}
	
	\noindent Let us consider a simple example of $M$ -- a sphere. It is easy to see that the constructed curve is indeed a geodesic. Moreover, from the two possible geodesics (given by the great circle passing through the two given points) connecting two points on the sphere, we obtain the shorter one, whenever the points are not antipodal.\\

	\noindent The question asked by M. Denkowski, which we are investigating here, if answered positively in the case of semi-algebraic or subanalytic surfaces of class $\mathcal{C}^2$, would locally give the affirmative answer to Hardt's conjecture. Moreover, the general theory (as exposed e.g. in \cite{long}) says that every semi-algebraic or subanalytic set admits a decomposition (stratification) into such manifolds. Therefore, is would potentially give other possibilities to attack Hardt's hypothesis.
	
	\section{Counterexample}\label{cex}
	As it turns out, there are counterexamples to the statement that such construction gives geodesics. One can come up with qualitative counterexamples, where it is evident that we do not obtain a geodesic. We will give an example of a very simple surface where we can compute the constructed curve exactly and compare it with the geodesics.\\
	Let us see how the construction works on a cylinder. Let the cylinder have radius $R$, height $h$ and $0z$ as its axis (and its medial axis). Let the base of the cylinder lie on the plane $0xy$. Without loss of generality we can assume that we connect two points on the two opposite bases of the cylinder. Moreover, by rotating the cylinder we can assume that the lower point $X$ has coordinates $(-R,0,0)$. The upper point $Y$ has coordinates $(a,b,h)$ with $a^2+b^2=R^2$. That fully describes the system we are dealing with and one can see it on the illustration below. 
	
	
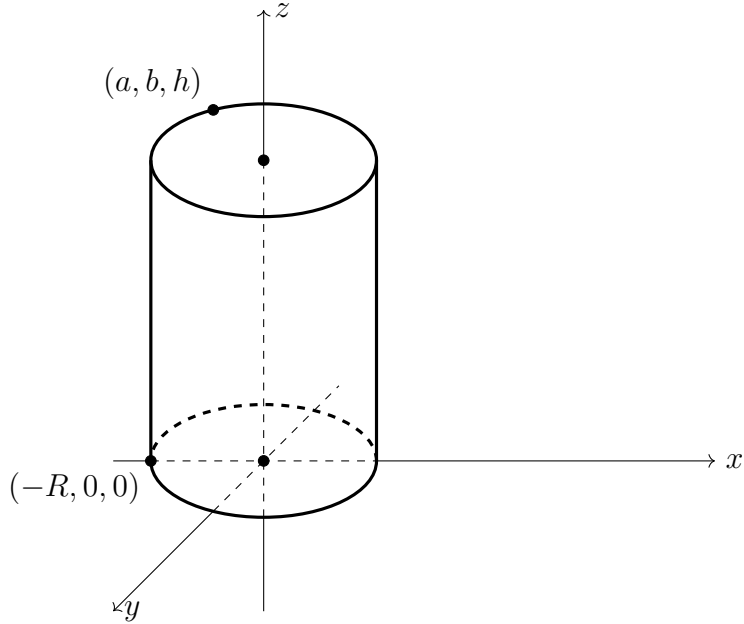
\begin{tikzpicture}
		
		\draw[very thick,dashed] (0,0) [partial ellipse=0:180:1.5cm and 0.75cm];
		\draw[very thick] (0,0) [partial ellipse=180:360:1.5cm and 0.75cm];
		\draw (-2,0) -- (-1.5,0);
		\draw[dashed] (-1.5,0) -- (1.5,0);
		\draw[->] (1.5,0) -- (6,0) node[anchor=west]{$x$};
		\draw[very thick] (0,4)  ellipse (1.5cm and 0.75cm);
		\draw[dashed] (0,-0.75) -- (0,4);
		\draw[->] (0,4) -- (0,6) node[anchor=west]{$z$};
		\filldraw[black] (0,0) circle (2pt);
		\filldraw[black] (0,4) circle (2pt);
		\draw[very thick] (1.5,0) -- (1.5,4);
		\draw[very thick] (-1.5,0) -- (-1.5,4);
		\draw (0,-0.75) -- (0,-2);
		\draw[dashed] (-0.67,-0.67) -- (1,1);
		\draw[->] (-0.67,-0.67) -- (-2,-2) node[anchor=west]{$y$};
		\filldraw[black] (-1.5,0) circle (2pt) node[anchor=north east]{$(-R,0,0)$};
		\filldraw[black] (-0.67,4.67) circle (2pt) node[anchor=south east]{$(a,b,h)$};
	\end{tikzpicture}

	\noindent We want to somehow describe the constructed curve. Every point $T$ on the segment $[X,Y]$ is uniquely determined by its height (i.e. the coordinate $z$). Let us denote it by $t$. Consider the section of the cylinder by the plane at the height $t$. Clearly, the point $T$ is projected onto the intersection of a ray with direction $T-(0,0,t)$, with the surface of the cylinder, as long as points $(-R,0)$ and $(a,b)$ are not antipodal (which corresponds to the fact that the segment $[X,Y]$ does not intersect the axis of the cylinder, i.e. the $z$-axis. Simple calculations show that the coordinates of $T$ are $$\bigg(-R+(a+R)\dfrac{t}{h},b\dfrac{t}{h},t\bigg).$$
	Therefore, its projection $T'$ onto the surface has coordinates 
	{\small
		$$\Bigg(R\dfrac{-R+(a+R)\dfrac{t}{h}}{\sqrt{\bigg(-R+(a+R)\dfrac{t}{h}\bigg)^2+\bigg(b\dfrac{t}{h}\bigg)^2}},R\dfrac{b\dfrac{t}{h}}{\sqrt{\bigg(-R+(a+R)\dfrac{t}{h}\bigg)^2+\bigg(b\dfrac{t}{h}\bigg)^2}},t\Bigg).$$
	}
	
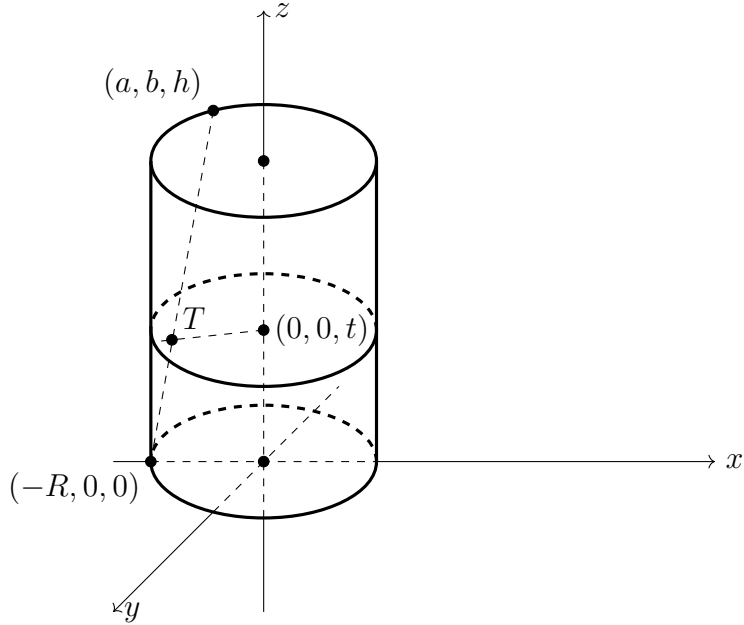
	\begin{tikzpicture}
		
		\draw[very thick,dashed] (0,0) [partial ellipse=0:180:1.5cm and 0.75cm];
		\draw[very thick] (0,0) [partial ellipse=180:360:1.5cm and 0.75cm];
		\draw (-2,0) -- (-1.5,0);
		\draw[dashed] (-1.5,0) -- (1.5,0);
		\draw[->] (1.5,0) -- (6,0) node[anchor=west]{$x$};
		\draw[very thick] (0,4)  ellipse (1.5cm and 0.75cm);
		\draw[dashed] (0,-0.75) -- (0,4);
		\draw[->] (0,4) -- (0,6) node[anchor=west]{$z$};
		\filldraw[black] (0,0) circle (2pt);
		\filldraw[black] (0,4) circle (2pt);
		\draw[very thick] (1.5,0) -- (1.5,4);
		\draw[very thick] (-1.5,0) -- (-1.5,4);
		\draw (0,-0.75) -- (0,-2);
		\draw[dashed] (-0.67,-0.67) -- (1,1);
		\draw[->] (-0.67,-0.67) -- (-2,-2) node[anchor=west]{$y$};
		\filldraw[black] (-1.5,0) circle (2pt) node[anchor=north east]{$(-R,0,0)$};
		\filldraw[black] (-0.67,4.67) circle (2pt) node[anchor=south east]{$(a,b,h)$};
		\draw[very thick,dashed] (0,1.75) [partial ellipse=0:180:1.5cm and 0.75cm];
		\draw[very thick] (0,1.75) [partial ellipse=180:360:1.5cm and 0.75cm];
		\filldraw[black] (0,1.75) circle (2pt) node[anchor=west]{$(0,0,t)$};
		\draw[dashed] (-1.5,0) -- (-0.67,4.67);
		\draw[dashed] (0,1.75) -- (-1.4,1.6);
		\filldraw[black] (-1.22,1.62) circle (2pt) node[anchor= south west]{$T$};
	\end{tikzpicture}

	\noindent Now we should check whether the constructed curve is geodesic. The natural approach is to use cylindrical coordinates. In such a coordinate system geodesics have a defining property: the angle depending on the height in an affine way. It is equivalent to saying that the derivative of the angle (as a function of the height, let us denote it $\alpha(t)$) is constant. The sine of the angle could be calculated from the coordinates of $T'$: $$\sin\alpha(t)=\dfrac{-R+(a+R)\dfrac{t}{h}}{\sqrt{\bigg(-R+(a+R)\dfrac{t}{h}\bigg)^2+\bigg(b\dfrac{t}{h}\bigg)^2}}.$$
	Then we can find $(\sin\alpha(t))'$, which we will rewrite as $\cos\alpha(t)\cdot \alpha'(t)$, where the cosine is calculated in the same manner as the sine. Therefore, the angle derivative is
	$$\alpha'(t)=\dfrac{bR}{h}\dfrac{1}{\bigg(-R +(R+a)\dfrac{t}{h}\bigg)^2+\bigg(b\dfrac{t}{h}\bigg)^2}.$$
	We conclude that it is constant only in the case when $a=-R$, $b=0$. That corresponds to the trivial case when the shortest path between two points is just the segment between them.

	\section{When the construction does work}
	
	Let us look at the next figure illustrating the result of construction described before. 
	
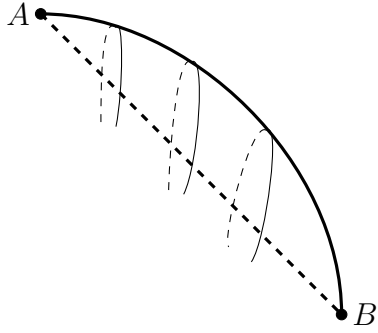
\begin{center}
		\begin{tikzpicture}
			\draw[very thick, dashed] (0,4) -- (4,0);
			\draw[very thick] (0,4) .. controls (2,4) and (4,2) .. (4,0);
			\filldraw[black] (0,4) circle (2pt) node[anchor=east]{$A$};
			\filldraw[black] (4,0) circle (2pt) node[anchor=west]{$B$};
			\draw (1,3.82) .. controls (1.1,4) and (1.1,3).. (1,2.5);
			\draw[dashed] (1,3.82) .. controls (0.8,4) and (0.8,3) .. (0.8,2.5);
			\draw (2,3.36) .. controls (2.2,3.5) and (2.1,2).. (1.9,1.6);
			\draw[dashed] (2,3.36) .. controls (1.8,3.5) and (1.7,2) .. (1.7,1.6);
			\draw (3,2.45) .. controls (3.2,2.6) and (3,1).. (2.8,0.7);
			\draw[dashed] (3,2.45) .. controls (2.8,2.6) and (2.4,1) .. (2.5,0.9);
			
		\end{tikzpicture}
	\end{center}

	\noindent At the first glance it appears that the constructed curve is a geodesic. There is a good reason for that. The picture suggests that the constructed curve $\gamma$ and the segment $[A,B]$ are coplanar. This is a sufficient condition for a curve to be a geodesic. The reason for this is the following: after reparametrization by arc-length (i.e. $||\gamma'||\equiv 1$, such a parametrization is possible for curves of class $\mathcal{C}^1$) the acceleration vector is perpendicular to the curve and lies in the plane $(ABX)$. Therefore, the acceleration is pointing directly at the point which was `projected' onto $X$. Therefore, the acceleration is perpendicular to the tangent space at the point $X$ to the respective sphere. However, that tangent space coincides with the tangent space to the initial manifold, therefore we get the perpendicularity of acceleration to the tangent space.
	
	
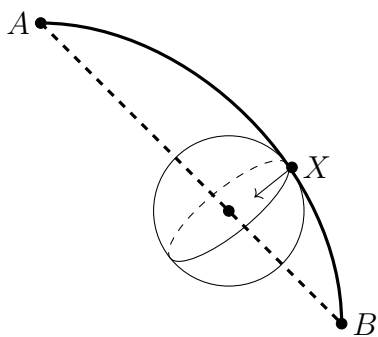
\begin{center}
		\begin{tikzpicture}
			\draw[very thick, dashed] (0,4) -- (4,0);
			\draw[very thick] (0,4) .. controls (2,4) and (4,2) .. (4,0);
			\filldraw[black] (0,4) circle (2pt) node[anchor=east]{$A$};
			\filldraw[black] (4,0) circle (2pt) node[anchor=west]{$B$};
			\filldraw[black] (2.5,1.5) circle (2pt);
			\draw (2.5,1.5) circle (1);
			\filldraw[black] (3.34,2.08) circle (2pt) node[anchor=west]{$X$};
			\draw[rotate around={39:(2.5,1.5)}] (2.5,1.5) [partial ellipse=180:360:1 and 0.3];
			\draw[dashed, rotate around={39:(2.5,1.5)}] (2.5,1.5) [partial ellipse=0:180:1 and 0.3];
			\draw[->] (3.34,2.08) -- (2.84,1.68);

		\end{tikzpicture}
	\end{center}

	\begin{theorem}\label{main}
		The described method gives geodesic if and only if the constructed curve and the segment $[A,B]$ are coplanar.
	\end{theorem}
	\begin{proof}
		We have already discussed the `if' part, hence we will concentrate on the second implication.\\
		Let us consider a two dimensional manifold $M$ in $\RR^3$ and points $A$, $B$ on that manifold. At every point $x$ of the segment $[A,B]$ we denote the distance from $x$ to $M$ by $r(x)$. From the assumptions we get that $B(x)=B(x,r(x))$ intersects with $M$ in exactly one point $m(x)$. Consider the so-called canal surface obtained as the boundary of the union of all the closed balls $B(x)$. Call it $K$. As it turns out, $K$ is a two dimensional $\mathcal{C}^1$ manifold with possible singularities at the points $A$ and $B$. That is so because $r'(x)^2<1$ (we will prove that later, see the next section). One can find a proof of the fact that $r'(x)^2<1$ implies the smoothness of $K$ in \cite{kanaly}.
		
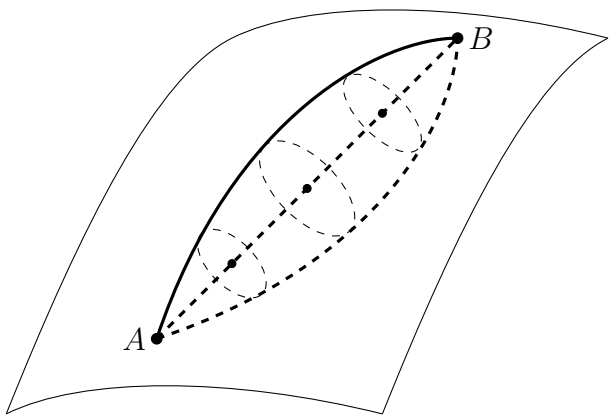
		\begin{center}
			\begin{tikzpicture}
				\draw[very thick, dashed] (0,0) -- (4,4);
				\draw[very thick] (0,0) .. controls (1,3) and (3,4) .. (4,4);
				\draw[very thick, dashed] (0,0) .. controls (3,1) and (4,3) .. (4,4);
				\filldraw[black] (0,0) circle (2pt) node[anchor=east]{$A$};
				\filldraw[black] (4,4) circle (2pt) node[anchor=west]{$B$};
				\draw[dashed,rotate around={135:(1,1)}] (1,1) ellipse (0.57 and 0.3);
				\draw[dashed,rotate around={135:(2,2)}] (2,2) ellipse (0.8 and 0.4);
				\draw[dashed,rotate around={135:(3,3)}] (3,3) ellipse (0.67 and 0.3);
				\filldraw[black] (1,1) circle (1.5pt);
				\filldraw[black] (2,2) circle (1.5pt);
				\filldraw[black] (3,3) circle (1.5pt);
				\draw (-2,-1) .. controls (-1,-0.5) and (1,-0.5) .. (3,-1);
				\draw (3,-1) .. controls (4,1.5) and (5,3.5) .. (6,4);
				\draw (6,4) .. controls (4,4.5) and (2,4.5) .. (1,4);
				\draw (-2,-1) .. controls (-1,1.5) and (0,3.5) .. (1,4);
				
			\end{tikzpicture}
		\end{center}

		\noindent The main idea of the proof now is the following: instead of checking whether the curve is a geodesic for some complicated surface $M$ we can check whether it is a geodesic on $K$. It follows from the fact that a curve is a geodesic if and only if the acceleration of a point moving with constant speed along that curve is perpendicular to the tangent space. For every point $m(x)$ the tangent spaces to $M$ and $K$ coincide. Therefore, the invoked geodesicity criterion is the same for $M$ and $K$.

		\bigskip
		
		\noindent The manifold $K$ is easier to study, as it is a surface of revolution. Let us investigate how can a geodesic approaching the `endpoints' of a revolution surface look like. We claim that any geodesic approaching $A$ is a `trivial' one. By that we mean that it is the intersection of $K$ and some plane containing $[A,B]$. Such geodesics are called {\it meridians}. That will end the proof of the theorem.

		\bigskip
		
		\noindent We will prove this claim using Clairaut's formula. It says that the product $r\sin{\theta}$ is constant along a geodesic where $r$ is the distance to the rotation axis and $\theta$ is the angle formed by the curve with the meridian. It implies that the distance from the points on the geodesic to the revolution axis cannot tend to $0$. Therefore, the geodesic cannot reach the point $A$.
		
	\end{proof}

	\bigskip
	
	\noindent We can use the proven theorem in the case of a cylinder. It says that we will get a geodesic iff the axis of the cylinder and the segment between the points $A$, $B$ are coplanar. This happens only when the angle coordinates are equal or opposite (in natural cylindrical coordinates). If the angles are opposite, then the segment between them intersects the axis, which is ruled out by the construction. Therefore, the only case is when the points have the same angle, which agrees with calculations made earlier.

	\section{Technical details}
	
	\noindent Now we will prove that $r'(x)^2<1$. First of all, we will show that locally, $r$ satisfies the inequality $$|r(x')-r(x)|\leqslant k(x)||x'-x||$$ for some $k(x)<1$. That means that $r$ increases with a rate strictly smaller then $1$, therefore it cannot decrease with a rate bigger or equal to $1$. That will complete the proof.\\
	Let us take points $x$ and $x'$ on the segment $[A,B]$ such that $x'$ is close to $x$. Points $m(x)$ and $m(x')$ are the closest ones to $x,x'$ in $M$, respectively. We know that $||m(x)-x'||\geqslant ||m(x')-x'||=r(x')$, so that $$r(x')-r(x)\leqslant ||m(x)-x'||-r(x)$$
	and 
	\begin{align*}
		||m(x)-x'||^2&=||(m(x)-x)+(x-x')||^2=\\
		&=||m(x)-x||^2+||x-x'||^2+2\langle m(x)-x,x-x'\rangle=\\
		&=r(x)^2+||x-x'||^2+2\langle m(x)-x,x-x'\rangle
	\end{align*}
	Thus,
	\begin{align*}
		||m(x)-x'||-r(x)&=\frac{||x-x'||^2+2\langle m(x)-x,x-x'\rangle}{||m(x)-x'||+r(x)}=\\
		&=||x-x'||\cdot \frac{||x-x'||+2\left\langle m(x)-x,\frac{x-x'}{||x-x'||}\right\rangle}{||m(x)-x'||+r(x).}
	\end{align*}
	Note that $\frac{x-x'}{||x-x'||}$ has the same direction as $AB$, therefore it is perpendicular to $m(x)-x$. Thus the fraction in the last expression tends to $0$ as $x'$ tends to $x$ (the denominator in the limit is $2r(x)$). As a result, sufficiently close to $x$ we get a bound by some constant $0<k(x)<1$.
	
	
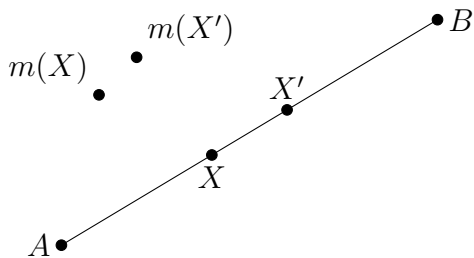
\begin{center}
		\begin{tikzpicture}
			\draw (0,0) -- (5,3);
			\filldraw[black] (0,0) circle (2pt) node[anchor=east]{$A$};
			\filldraw[black] (5,3) circle (2pt) node[anchor=west]{$B$};
			\filldraw[black] (2,1.2) circle (2pt) node[anchor=north]{$X$};
			\filldraw[black] (3,1.8) circle (2pt) node[anchor=south]{$X'$};
			\filldraw[black] (0.5,2) circle (2pt) node[anchor=south east]{$m(X)$};
			\filldraw[black] (1,2.5) circle (2pt) node[anchor=south west]{$m(X')$};
		\end{tikzpicture}
	\end{center}

	\section{Generalization to higher dimensions}
	
	\noindent Now we will comment on the Clairaut formula. The classical statement is for two-dimensional surfaces of revolution in $\RR^3$. In order to generalize Theorem \ref{main} to higher dimensions we first need to generalize the formula to higher dimensional case. First of all, let us understand how to generalize the notion of a surface of revolution. What we want from such generalization is to encapsulate surfaces $K$ obtained in the proof of the first theorem. To be more precise, $\mathcal{C}^1$ manifolds obtained as a boundary of the union of balls centered at the points of a given segment. A natural candidate is the following.
	
	\begin{definicja}
		A hyper-surface in $\RR^n=\RR\times \RR^{n-1}$ is called a \textit{surface of revolution} if it can be described as $$K=\big\{(x,y)\in I\times \RR^{n-1}| \quad||y||^2=R(x)^2\big\}$$
		for some differentiable function $R\colon I \to \RR$, where $I\subset\RR$ is a segment. The line $\RR \times \{0\}^{n-1}$, sometimes just the segment $I\times\{0\}^{n-1}$, are called the \textit{axis} of the respective surface of revolution.
	\end{definicja} 
	
	\noindent Notice that the function $R$ in the above definition and the function $r$ from the proof of Theorem \ref{main} generally are not equal (they coincide only if they are constant). As $K$ is the set of zeroes of a function with gradient $2(R(x)R'(x),y)$, we deduce that it is always a submersion of the same class as $R$. Therefore, $K$ is a submanifold of $\RR^n$ of codimension $1$.\\
	Moreover, the manifold $K$ obtained in the proof of the Theorem \ref{main} is a surface of revolution. One way to see that is to note that the intersection of an arbitrary ball $B(x)$ with a hyperplane perpendicular to the axis passing through $y$ is an $(n-1)$-dimensional ball centered at $y$. Thus, we get that the boundary of the union of such intersections is an $(n-2)$-dimensional sphere centered at $y$ -- exactly what we would expect from a surface of revolution.\\

	\noindent The next notion we need to define is that of a meridian on such surfaces of revolution. The next definition seems to be the most natural.
	
	\begin{definicja}
		A curve on the surface of revolution $K$ is called a \textit{meridian} if it is an intersection of $K$ with some plane passing through the axis of $K$.
	\end{definicja}
	
	\noindent Notice that for each point on the surface of revolution there is exactly one meridian passing through it.\\
	Let us give a proof a generalization of Clairaut's formula in the situation considered, inspired by \cite{clairaut}.
	
	\begin{theorem}[Generalized Clairaut's formula]\label{clairaut}
		For a geodesic $\gamma(t)$ on a surface of revolution $K$ the expression $R(\gamma(t)) \cdot \sin(\angle(\gamma'(t),\text{mer}(\gamma(t))))$ is constant, where $R$ is the distance to the axis of $K$ and $\text{mer}(p)$ is the vector tangent to the meridian passing through $p$ in the positive direction. In particular, the expression $\gamma(t)\wedge \gamma '(t) \wedge \vec{x}$ is constant.
	\end{theorem}
	
	
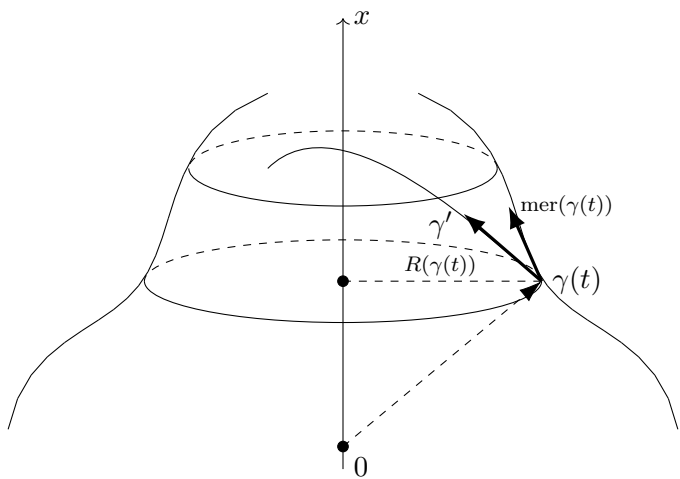
\begin{center}
		\begin{tikzpicture}
			\draw[->] (0,0) -- (0,6) node[anchor=west]{\small $x$};
			\filldraw[black] (0,0.3) circle (2pt) node[anchor=north west]{\small $0$}; 
			\bezierq{\pgfpointxy{1}{5}}{\pgfpointxy{3.5}{4}}{\pgfpointxy{1}{1}}{\pgfpointxy{4.5}{3}}{\pgfpointxy{4.5}{0}};
			\bezierq{\pgfpointxy{-1}{5}}{\pgfpointxy{-3.5}{4}}{\pgfpointxy{-1}{1}}{\pgfpointxy{-4.5}{3}}{\pgfpointxy{-4.5}{0}};
			\draw[dashed] (0,4) [partial ellipse=0:180:2.05 and 0.5];
			\draw (0,4) [partial ellipse=180:360:2.05 and 0.5];
			\draw[dashed] (0,2.5) [partial ellipse=0:180:2.64 and 0.55];
			\draw (0,2.5) [partial ellipse=180:360:2.64 and 0.55];
			\filldraw[black] (0,2.5) circle (2pt); 
			\draw[dashed, -{Latex[length=3mm,width=2mm]}] (0,0.3) -- (2.64,2.5);
			\draw[dashed] (0,2.5) -- (2.64,2.5) node[anchor=west]{\small $\gamma(t)$};
			\node at (1.3,2.7) {\tiny $R(\gamma(t))$};
			\draw[very thick,-Latex] (2.64,2.5) -- (2.2,3.5) node[anchor= west]{\tiny $\text{mer}(\gamma(t))$};
			\draw (2.64,2.5) .. controls (2,3) and (0,5) .. (-1,4);
			\draw[very thick,-Latex] (2.64,2.5) -- (1.6,3.4);
			\node at (1.3,3.25) {\small $\gamma'$};
		\end{tikzpicture}
	\end{center}

	\begin{proof}
		First, we will show that the expression $$\gamma(t)\wedge \gamma '(t) \wedge \vec{x}$$ is constant along the geodesic, where $\vec{x}$ is a unit vector directed in the positive direction along the axis of $K$. Then we will show that its value can be interpreted as $$R(\gamma(t)) \cdot \sin(\angle(\gamma'(t),\text{mer}(\gamma(t)))),$$
		which will end the proof.\\
		Let us calculate the derivative of the above mentioned expression and show that is equal to zero. We differentiate the product.
		\begin{align*}
			(\gamma(t)\wedge \gamma '(t) &\wedge  \vec{x})'=\\
			&=\gamma '(t) \wedge \gamma '(t) \wedge  \vec{x} + \gamma(t) \wedge \gamma ''(t) \wedge  \vec{x} + \gamma(t) \wedge \gamma '(t) \wedge 0=\\
			&=\gamma(t) \wedge \gamma ''(t) \wedge  \vec{x} =0.
		\end{align*}
		The last equality holds since the vectors $\gamma(t)$, $\gamma ''(t)$ and $ \vec{x}$ lie in one plane. That is so because $\gamma ''(t)$ is a vector in the normal space to $K$ (as $\gamma$ is a geodesic).\\
		In order to understand how the normal space to $K$ looks like, let us consider an $(n-1)$-dimensional affine hyperplane $S$ passing through $\gamma(t)$ and perpendicular to the axis of $K$. Let us denote the intersection of $S$ and $K$ by $s$. In this case we get that the orthogonal projection of  the vector $\gamma''(t)$ on $s$ is perpendicular to the normal space to $s$. We know that $s$ is an $(n-2)$-dimensional sphere centered on the axis of $K$, thus the  radius in the direction $\gamma''(t)$ from the point $\gamma(t)$ either intersects the axis of $K$, or is parallel to it. 
		
		
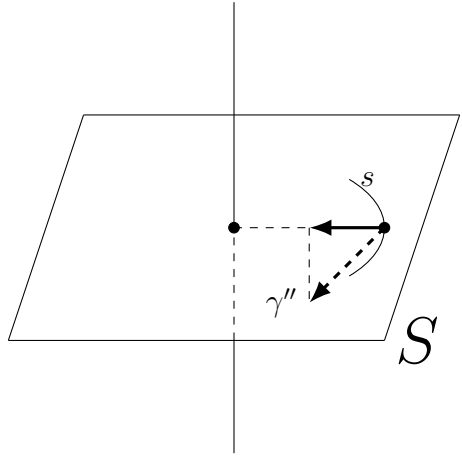
\begin{center}
			\begin{tikzpicture}
				\draw (0,0) -- (0,1.5);
				\draw[dashed] (0,1.5) -- (0,3);
				\draw (0,3) -- (0,6);
				\draw (-3,1.5) -- (2,1.5)  node[anchor=west]{\huge{$S$}};
				\draw (2,1.5) -- (3,4.5);
				\draw (3,4.5) -- (-2,4.5);
				\draw (-2,4.5) -- (-3,1.5);
				\filldraw[black] (0,3) circle (2pt);
				\filldraw[black] (2,3) circle (2pt) ;
				\draw (0,3) [partial ellipse=320:400:2 and 1] node[anchor=west]{{$s$}};
				\draw[dashed] (0,3) -- (2,3);
				\draw[very thick,-Latex] (2,3) -- (1,3);
				\draw[dashed,very thick,-Latex] (2,3) -- (1,2)node[anchor=east]{{$\gamma''$}};
				\draw[dashed] (1,3) -- (1,2);
			\end{tikzpicture}
		\end{center}

		\noindent Therefore, we indeed get that the vectors $\gamma(t)$, $\gamma ''(t)$ and $ \vec{x}$ lie in the plane passing through the axis of $K$ and the point $\gamma(t)$.\\
		In order to obtain the desired property, we note that the expression $$||\gamma(t)\wedge \gamma '(t) \wedge \vec{x}||$$ is constant and equal to the volume of the parallelepiped spanned by the vectors $\gamma(t)$, $\gamma '(t)$ and $\vec{x}$. However, it is also equal to the area of the parallelogram spanned by vectors $\gamma(t)$ and $\vec{x}$ multiplied by the height of parallelepiped. Keeping in mind that $||\vec{x}||=||\gamma '(t)||=1$ we get that the area of the parallelogram is $R(\gamma(t))$ and the height is $\sin(\angle(\gamma'(t),\text{mer}(\gamma(t))))$ (because $\text{mer}(\gamma(t))))$ is in the plane spanned by $\gamma(t)$ and $\vec{x}$). This gives us the desired expression.
		
		
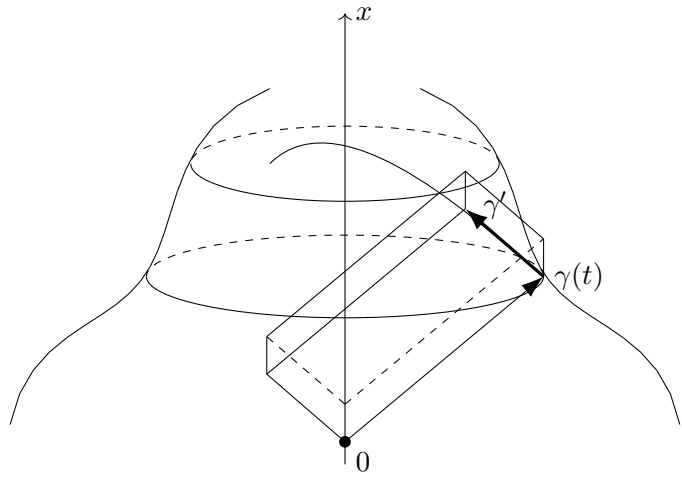
\begin{center}
			\begin{tikzpicture}
				\draw[->] (0,0) -- (0,6) node[anchor=west]{\small $x$};
				\filldraw[black] (0,0.3) circle (2pt) node[anchor=north west]{\small $0$}; 
				\bezierq{\pgfpointxy{1}{5}}{\pgfpointxy{3.5}{4}}{\pgfpointxy{1}{1}}{\pgfpointxy{4.5}{3}}{\pgfpointxy{4.5}{0}};
				\bezierq{\pgfpointxy{-1}{5}}{\pgfpointxy{-3.5}{4}}{\pgfpointxy{-1}{1}}{\pgfpointxy{-4.5}{3}}{\pgfpointxy{-4.5}{0}};
				\draw[dashed] (0,4) [partial ellipse=0:180:2.05 and 0.5];
				\draw (0,4) [partial ellipse=180:360:2.05 and 0.5];
				\draw[dashed] (0,2.5) [partial ellipse=0:180:2.64 and 0.55];
				\draw (0,2.5) [partial ellipse=180:360:2.64 and 0.55]; 
				\draw[-{Latex[length=3mm,width=2mm]}] (0,0.3) -- (2.64,2.5) node[anchor=west]{\small $\gamma(t)$};
				\draw (2.64,2.5) .. controls (2,3) and (0,5) .. (-1,4);
				\draw[very thick,-Latex] (2.64,2.5) -- (1.6,3.4);
				\node at (2,3.45) {\small $\gamma'$};
				\draw (0,0.3) -- (-1.04,1.2);
				\draw (1.6,3.4) -- (-1.04,1.2);
				\draw (-1.04,1.2) -- (-1.04,1.7);
				\draw[dashed] (0,0.8) -- (-1.04,1.7);
				\draw[dashed] (0,0.8) -- (2.64,3);
				\draw (2.64,3) -- (2.64,2.5);
				\draw (2.64,3) -- (1.6,3.9);
				\draw (1.6,3.4) -- (1.6,3.9);
				\draw (-1.04,1.7) -- (1.6,3.9);
			\end{tikzpicture}
		\end{center}

	\end{proof}

	\noindent The above preparation gives us a way to generalize Theorem \ref{main} to the case of $(n-1)$-dimensional surfaces in $\RR^n$. The proof is exactly the same as in the three dimensional case, therefore we leave it to the reader.


\newpage
	\addcontentsline{toc}{section}{Bibliografia}
	
	\begin{thebibliography}{99}
		\bibitem {Hardt} R. M. Hardt. Some analytic bounds for subanalytic sets. \textit{Progress in mathematics}, 27: 259-267, 1983.
		\bibitem{long} 
		Z. Denkowska, M. P. Denkowski,	
		A long and winding road to definable sets, {\it Journal of Singularities} vol. 13, 57-86, 2015.
		\bibitem {Szkielety} Maciej P. Denkowski. When the medial axis meets the singularities. \textit{Analytic and Algebraic Geometry}, 3: 41-66, 2019.  
		\bibitem {continuity} W. D. Evans,  D. J. Harris. Sobolev Embeddings for Generalized Ridged Domains. \textit{Proceedings of the London Mathematical Society}, s3-54(1), 141–175, 1987.
		\bibitem {KO} K. Kurdyka, P. Orro, \textit{Distance g\'eod\'esique sur un sous-analytique}, Rev. Mat. Univ. Complut. Madrid 10 (Special Issue, suppl.): 173-182. 1997
		\bibitem {kanaly} Martin Peternell, Helmut Pottmann. Computing Rational Parametrizations of Canal Surfaces. \textit{Journal of Symbolic Computation}, 23 (2–3): 255-266, 1997.
		\bibitem {clairaut} M\'arcio Rosa, A Pedestrian Approach to the Clairaut Relation, Class notes {\it IME Unicamp}, 2016.
	\end{thebibliography}
\end{document}